\documentclass[a4paper,12pt,oneside]{article}

\usepackage[british,UKenglish,USenglish,english,american]{babel}
\usepackage{geometry}
\usepackage{stmaryrd}
\usepackage{xcolor}   
\usepackage{amsmath}
\usepackage{amsthm}
\usepackage{amssymb}
\usepackage{mathrsfs}
\usepackage{mathtools}
\usepackage{graphicx}
\usepackage{amssymb}
\usepackage{epstopdf}
\usepackage{enumerate}
\usepackage{bbm}
\usepackage{braket}
\usepackage{comment}

\newtheorem{theorem}{Theorem}[section]
\newtheorem{definition}{Definition}[section]
\newtheorem{corollary}{Corollary}[section]

\theoremstyle{remark}
\newtheorem{remark}{Remark}[section]
\theoremstyle{definition}
\newtheorem{example}{Example}

\newcommand\abs[1]{\left|#1\right|}

\providecommand{\keywords}[1]{\textbf{Keywords: } #1}

\newcommand{\E}{\mathbb{E}}
\newcommand{\e}{\epsilon}
\newcommand{\R}{\mathbb{R}}
\newcommand{\U}{\mathcal{U}_0}

\title{Feedback optimal controllers for the Heston model}
\author{Viorel Barbu  - A.I. Cuza University, Iasi, Romania\\ \and 
Chiara Benazzoli - Dept. of Mathematics, University of Trento \\ \and Luca Di Persio - Dept. of Computer Science, University of Verona}

\date{}

\numberwithin{equation}{section}

\newcommand{\norm}[1]{\left\lVert#1\right\rVert}

\begin{document}

\maketitle 


\begin{abstract} We prove the existence of an optimal feedback controller for a stochastic optimization problem constituted by a variation of the Heston model, where a stochastic input process is added in order to minimize a given performance criterion.
The stochastic feedback controller is searched by solving a nonlinear backward parabolic equation for which one proves the existence of a martingale solution.
\end{abstract}

\section{Introduction}

Stochastic volatility models (SVMs) are widely used within a pletora of  financial settings, spanning from the risk sector, to the interest rates one, from econometric problems, to insurance ones, see, e.g., \cite{CordoniDipersioBackward, CordoniDipersioVasicek, Grzelak, HaastrechtPelsser,Jones, YiLiViensZeng}, and references therein.
In fact, SVMs allow for a finer analysis of relevant time series as they appear in the real-world financial arenas. Indeed, daily return data series show two peculiarities among different types of assets in different markets and in different periods, namely the \emph{volatility clustering phenomenon} 
and the \emph{fat-tailed and highly peaked distribution} relative to the normal distribution, as assumed, e.g., in the celebrated Black ad Scholes model.

Volatility clustering refers to the fact that \emph{large changes tend to be followed by large changes, of either sign, and small changes tend to be followed by small changes} (see \cite{Mandelbrot}). The latter means that, over a significant time window, it can be noted the presence of both high volatility period and low volatility one, separately, rather than a constant average level of volatility persisting over time.

Such peculiarities can be captured by SVMs because they are characterized by a volatility term which is itself a stochastic process. 
As an example, consider the following  price dynamic for given asset $S$, driven by a Geometric Brownian motion
\begin{equation}
\label{eq:SDES}
dS_t=\mu_t S_t\,dt+\sqrt{\nu_t}S_t\,dW_t
\end{equation}
where
\begin{equation}
\label{eq:SDEnu}
d\nu_t=k_t\,dt+\sigma_t\,dB_t
\end{equation}
and where $W$ and $B$ denote two Brownian motions with correlation $\rho$.
In this case, the volatility of the price is no longer a deterministic function of $S$, but it is itself randomly distributed. Different choices for \eqref{eq:SDEnu} allow for a multitude of models that  properly represent different financial data, see, e.g, \cite{Barndorff,CIR,Hull,Scott,Wiggins}.
In this context we focus our attention on the Heston model, firstly introduced in \cite{Heston} to price European bond and currency options, aiming at generalize and  overcome the biases of the Black ad Scholes model. We recall that the Heston model assumes that the asset price follows the dynamic in \eqref{eq:SDES}, where $\nu$ follows the Ornstein-Uhlenbeck process
\[
d\nu_t=k(\theta-\nu_t)\,dt+\sigma\sqrt{\nu_t}\,dB_t\;,
\] 
explaining the volatility smile.
In what follows we focus our attention on a controlled version of the Heston model, namely adding a control component w.r.t. to the volatility term appearing in \eqref{eq:SDES}. This approach is mainly intended to take in account exogenous influence of external financial actors within a given investment setting. This is the case
of the possible action of a Central Bank which aims at minimizing the probability of abrupt changes within markets where a relevant group of banks are exposed.
In this scenario, the role played by the  Central Bank can be realized putting money in the market by means of actions. Recently, such type of wide breath action has been concretized by the European Central Bank which has adopted the so called quantitative easing monetary policy. In this case, the volatility control has been implemented buying a predetermined amount of financial assets emitted mainly  from (national) commercial banks . This results in: a rise of the prices of the interested  financial assets, and a  raise of interested assets prices lowering their yield and simultaneously increasing the money supply, with the final result of a drastic reduction of the volatility terms due to the instability of some European regions, in general, and concerning particular, highly exposed, banks of well determined states.

Analogously one can see the aforementioned control issue from the point of view of maximizing  the expected discounted utility of consumption and terminal wealth, as, e.g., made in \cite{ChangRong} w.r.t. a fixed finite investment horizon. 
On the other hand, as pointed out in \cite{Kraft}, solving   the portfolio problem for Heston's stochastic volatility model may lead  that  partial equilibrium can be obtained  only under a specific condition on the model parameters, particularly w.r.t. the specification of the market price of risk, which means that  the market price of risk has to  be specified with great care.

The main aim of this paper is to study an optimization problem where the performance function we want to minimize depends on a stochastic process whose dynamic is modeled through a \emph{controlled} variation of the Heston model. Namely, we assume that a controller is added in the volatility component in eq.~\eqref{eq:SDES}. 
The present work is structured as follows: in Section~\ref{Sec:HestonModel} we give the necessary details to introduce the problem, also providing the well-posedness of this stochastic model; in Section~\ref{Sec:dpe} we introduce the associated Hamilton-Jacobi equation which is reduced to a nonlinear parabolic equation on $(0,\infty)\times(0,\infty)$. The main result of this paper is the existence and uniqueness of solutions for this equation. In Section~\ref{sec:optimalfeedback} we derive the existence of an optimal controller in feedback form. 

\section{The Heston control model}\label{Sec:HestonModel}
Let $(\Omega,\mathcal{F},\mathbb{P})$ be a probability space and $W_1$, $W_2$ be Brownian motions. Let $(\mathcal{F}_t)_{t\ge0}$ be the natural filtration generated by $W:=(W_1,W_2)$.
Consider the controlled stochastic system
\begin{equation}
\label{sys:model}
\begin{cases}
dX_1=\mu X_1\,dt+X_1\sqrt{u X_2}\,dW_1\,,\quad t\in(0,T)\\
dX_2=k(\theta-X_2)\,dt+\sigma\sqrt{X_2}\,dW_2\,,\quad t\in(0,T)\\
X_1(0)=X_1^0\,,\quad X_2(0)=X_2^0\,,
\end{cases}
\end{equation}
where $X^0_i\ge0$ $i=1,2$ and $W_1$,$W_2$ are 1D-Brownian motions.
Here $\mu$, $\kappa$, $\sigma$, $\theta$ are positive constants.

We note that \eqref{sys:model}  extends the classical Heston model, by adding a controlled component in the volatility of the primary stochastic process $X_1$.  
The control parameter $u$ is a $(\mathcal{F}_t)_{t\ge0}$-adapted stochastic process $u:[0,T]\to\mathbb{R}$ which is assumed to take values in the interval $U=[a,b]$, where $0<a<b<\infty$.

We associate with the control system \eqref{sys:model} a performance function
\begin{equation}
\label{eq:performance}
J(u;X_1^0,X_2^0)=\E\int_0^T X_1^2(t)f(X_1(t),u(t))\,dt+\E g(X_1(T))
\end{equation}
and denote by $\mathcal{U}_0$ the class of stochastic control processes $u:[0,T]\to U$.\\

The function $f:\R\times\R\to\R$ is assumed to satisfy the following hypotheses.
\begin{enumerate}[(i)]
\item  $f$ is continuous on $\R^2$ and for each $x\in\R$, $u\mapsto f(x,u)$ is convex. Moreover, $\inf\{f(x,u); u\in[a,b]\}=0$ for all $x\in\R$.
\end{enumerate}
In the following we denote by $f_u$ the sub-differential of the function $u\mapsto f(\cdot,u)$.

The optimal control problem we consider here is the following:
\begin{description}
\item[(P)] \label{pbm} Minimize $J(u;X^0_1,X^0_2)$ on the set of all stochastic processes $u\in\U$ satisfying $\eqref{sys:model}$.
\end{description}
Our objective  is to find an optimal feedback controller $u(t)$ for problem (P) via the Hamilton-Jacobi (dynamic programming) equation.

A first problem regards the well posedness of state system \eqref{sys:model}. It is clear that the strong solution $X_2$, if exists, should be found in class of nonnegative processes on $[0,T]$. However since the function $x\to\sqrt{x}$ is not Lipschitz it is not clear if such a solution exists for all $X^0_2\ge0$. However, we have the following existence result for a martingale solution $X=(X_1,X_2)$ to \eqref{sys:model}.
\begin{theorem}\label{thm:existence}
Assume that 
\begin{equation}
\label{eq:ic}
k\theta\ge\frac{1}{2}\sigma^2\,,\quad X^0_2\ge0\,.
\end{equation}
Then there is at least one martingale solution $X=(X_1,X_2)$ for \eqref{sys:model} in \\$(L^2(\Omega;C([0,T])))^2$ such that
\begin{equation}
\label{eq:pos}
X_1(t)\,,\,X_2(t)\ge0\,,\quad\forall t\in[0,T]\,,\quad\text{$\mathbb{P}$-a.s.}
\end{equation}
Moreover, $X_1(t)\ge0$, $\forall t\in[0,T]$, $\mathbb{P}$-a.s. and
\begin{equation}
\label{eq:boundX2}
\E\sup_{t\in[0,T]}\left(\abs{X_1(t)}^2+\abs{X_2(t)}^2\right)\le\bar C(\abs{X_1^0}^2+\abs{X_2^0}^2)+CT
\end{equation}
for suitable constants $\bar C,C>0$.
\end{theorem}
By martingale solution $X=(X_1,X_2)$ to \eqref{sys:model} we mean an $(\mathcal{F}_t)_{t\ge 0}$-adapted process $X$ in a filtered probability space $(\Omega, \mathcal{F},(\mathcal{F}_t)_{t\ge0})$ which is pathwise continuous on $[0,T]$ and satisfies the equations
\begin{equation}
\label{def:martsol}
\begin{cases}
X_1(t)=X^0_1+\int_0^t \mu X_1(s)\,ds+\int_0^t X_1(s)\sqrt{u(s) X_2(s)}\,dW_1(s)\,,\quad t\in(0,T)\,,\\
X_2(t)=X^0_2+\int_0^tk(\theta-X_2(s))\,ds+\int_0^t\sigma\sqrt{X_2(s)}\,dW_2(s)\,,\quad t\in(0,T)\,.
\end{cases}
\end{equation}
together with a Wiener process $W=(W_1,W_2)$.
\begin{proof}
We approximate the second equation in \eqref{sys:model} by
\begin{equation}
\label{eq:eps}
\begin{aligned}
&dX^\e_2=k(\theta-X^\e_2)\,dt+\sigma\frac{X^\e_2}{\sqrt{|X^\e_2|+\e}}\,dW_2\,,\quad t\in(0,T)\,,\\
&X^\e_2(0)=X^2_0\,,
\end{aligned}
\end{equation}
and associate to \eqref{eq:eps} the equation 
\begin{equation}
\label{eq:eps1}
\begin{cases}
dX^\e_1=\mu X^\e_1\,dt+X^\e_1\sqrt{u X^\e_2}\,dW_1\,,\quad t\in(0,T)\,,\\
X_1^\e(0)=X_1^0\,.
\end{cases}
\end{equation}

We shall prove first that under assumption \eqref{eq:ic}, we have 
\begin{equation}
\label{eq:pose}
X^\e_2(t)\ge0\,,\quad\forall t\in[0,T]\,,\quad\text{$\mathbb{P}$-a.s.}
\end{equation}
Now we apply It\^o's formula to function $\varphi(x)=\frac{1}{2}(x^-)^2$ for all $x\in \R$ where $x^-=\max\set{0,-x}$. We have
\[
\varphi'(x)=-x^-\,,\quad\varphi''(x)=H(-x), \quad\forall x\in\R\,,
\]
where $H$ is the Heaviside function. We get
\[
\begin{split}
\frac{1}{2}d\abs{(X_2^\e)^-(t)}^2=-k\abs{(X_2^\e)^-(t)}^2\,dt-k\theta (X_2^\e)^-(t)\,dt +\sigma\frac{(X^\e_2(t)^-)^2}{\sqrt{|X^\e_2(t)|+\e}}\,dW_2(t)\\+\frac{\sigma^2}{2}\frac{((X^\e_2)^-(t))^2}{(X^\e_2)^- (t)+\e}H(-X_2^\e(t))\,dt\,.
\end{split}
\]
By \eqref{eq:ic} this yields
\[
\frac{1}{2}\E\left[\abs{(X_2^\e)^-(t)}^2+k\int_0^t \abs{(X_2^\e)^-(s)}^2\,ds\right]\le\E\left[\int_0^t\left(\frac{\sigma^2}{2}-k\theta\right)\abs{(X_2^\e)^-(t)}\,ds\right]\le0
\]
for all $t\in[0,T]$, and therefore $(X^\e_2)^-(t)=0$ on $(0,T)\times\Omega$ which implies \eqref{eq:pose} as claimed. 

Substituting $X^\e_2$ in \eqref{eq:eps1}, we get that $(X^\e_1,X^\e_2)$ is a strong solution to \eqref{sys:model}. 
Indeed, if we represent $X^\e_1$ as 
\[
X^\e_1(t)=\exp\left(\int_0^t \sqrt{u(s)X^\e_2(s)}dW_1(s)\right)y_\e(t)\,,
\]
we have by It\^o's formula that
\begin{align*}
dX^\e_1(t)&=\exp\left(\int_0^t \sqrt{u(s)X^\e_2(s)}dW_1(s)\right)dy_\e(t)+\\&\hspace{1cm}
+\sqrt{u(t)X^\e_2(t)}\exp\left(\int_0^t \sqrt{u(s)X^\e_2(s)}dW_1(s)\right)y_\e(t)\,dW_1(t)+\\
&\hspace{1cm}
+\frac{1}{2}u(t)X^\e_2(t)\exp\left(\int_0^t \sqrt{u(s)X^\e_2(s)}dW_1(s)\right)y_\e(t)dt
\end{align*}
and substituting in \eqref{eq:eps1} we get for $y_\e$ the random differential equation
\begin{align*}
&\frac{dy_\e}{dt}=\left(\mu-\frac{1}{2}uX^\e_2\right)y_\e\,,\quad t\in(0,T)\\
&y(0)=X_1^0\,,
\end{align*}
which clearly has a unique $(\mathcal{F}_t)_{t\ge0}$-adapted solution $y_\e$.
Therefore,
\begin{equation}
\label{eq:Xe}
X^\e_1(t)=\exp\left(\int_0^t \left(\mu-\frac{1}{2}u(s)X^\e_2(s)\right)\,ds+\int_0^t \sqrt{u(s)X^\e_2(s)}dW_1(s)\right)X^0_1
\end{equation}
is together with $X^\e_2$ the solution to \eqref{eq:eps}-\eqref{eq:eps1}.
The uniqueness of such a solution is immediate. By \eqref{eq:Xe} it follows also that $X_1\ge 0$ as claimed.

By It\^o's formula we see that
\[
\begin{split}
\frac{1}{2}\abs{X_2^\e(t)}^2=\frac{1}{2}\abs{X_2^0}^2+\int_0^t k(\theta-X^\e_2(s))X_2^\e(s)\,ds+\frac{\sigma^2}{2}\int_0^t \frac{(X^\e_2(s))^2}{X^\e_2(s)+\e}\,ds+\\
+\int_0^t \frac{\sigma(X^\e_2(s))^2}{\sqrt{X^\e_2(s)+\e}}\,dW_2(s)
\end{split}
\]
and this yield via Burkholder-Davis-Gundy Theorem (see e.g. \cite{DaPrato})
\[
\E\sup\{\abs{X_2^\e(t)}^2\,;\,t\in[0,T]\}\le C_1(1+\abs{X_2^0}^2)\,.
\]
We have also by \eqref{eq:eps}
\[
\E\left[\abs{X_2^\e(t)-X_2^\e(s)}^2\right]\le C\E\left[\int_s^t(1+\abs{X_2^\e(r)}^2)\,dr\right]\le C_2(t-s)\quad\forall s,t\in[0,T]\,.
\]
Similarly by \eqref{eq:eps1}, we get
\begin{gather*}
\E\sup\{\abs{X_1^\e(t)}^2\,;\,t\in[0,T]\}\le C_3(1+\abs{X_1^0}^2)\,,\\
\E\left[\abs{X_1^\e(t)-X_1^\e(s)}^2\right]\le  C_4(t-s)\quad\forall s,t\in[0,T]\,.
\end{gather*}
Hence we have
\begin{gather}
\label{supXe} \E\sup\{\abs{X^\e(t)}^2\,;\,t\in[0,T]\}\le C_5(1+\abs{X_1^0}^2+\abs{X_2^0}^2)\,,\\
\label{stXe}\E\left[\abs{X^\e(t)-X^\e(s)}^2\right]\le  C_6(t-s)\quad\forall s,t\in[0,T]\,.
\end{gather}
for $X^\e=(X^\e_1,X^\e_2)$.

We set $\nu_\e=\mathscr{L}(X^\e)$,  there is, $\nu_\e(\Gamma)=\mathbb{P}(X^\e\in\Gamma)$ for each Borelian set $\Gamma\subset(C([0,T];\R))^2$ and note that $\{\nu_\e\}_{\e>0}$ is tight in $C([0,T];\R^2)$. 
This means that for each $\delta>0$ there is a complete set $\Gamma \subset (C([0,T];\R^2))$ such that $\nu_\e(\Gamma^c)\le\delta$ for all $\e>0$. We take $\Gamma$ of the form
\[
\begin{split}
\Gamma_{r,\gamma}=\{y\in C([0,T];\R^2)\,: \, \abs{y(t)}\le r\,, \forall t\in[0,T]\,, \quad\\
\abs{y(t)-y(s)}\le \gamma\abs{t-s}^\frac{1}{2}\,, \forall t,s\in[0,T]\}\,.
\end{split}
\]
Clearly by the Ascoli-Arzela theorem, $\Gamma_{r,\gamma}$ is compact in $C([0,T];\R^2)$. On the other hand, by \eqref{supXe}-\eqref{stXe} and the well known inequality
\[
\rho\mathbb{P}[\abs{Y}\ge\rho]\le\E\abs{Y}\,,\quad\forall\rho>0
\]
it follows that there are $r,\gamma$ independent of $\e$ such that $\nu_\e(\Gamma^c_{r,\gamma})\le\delta$ as claimed.
Then by the Skorohod's theorem there is a probability space $ (\tilde \Omega,\mathcal{\tilde F},\mathbb{\tilde P})$ and random variables $\tilde X,\tilde X^\e$ such that $\mathscr{L}(\tilde X^\e)=\mathscr{L}( X^\e)$ and $\mathbb{\tilde P}$-a.e. $\omega\in\tilde\Omega$\,,
\[
\tilde X_\e\to\tilde X\quad\text{in $C([0,T];\R^2)$.}
\]
Then we may pass to limit in \eqref{eq:eps}-\eqref{eq:eps1} and see that $\tilde X=(\tilde X_1,\tilde X_2)$ satisfies system \eqref{sys:} in the space $ (\tilde \Omega,\mathcal{\tilde F},\mathbb{\tilde P})$ for a new pair $\tilde W=(\tilde W_1,\tilde W_2)$ of Wiener processes in this space. See \cite{Barbuetal} for details. This completes the proof of existence of a martingale solution. Clearly \eqref{eq:pos},\eqref{eq:boundX2} hold for this solution which in the following we shall denote again $X=(X_1,X_2)$, $W=(W_1,W_2)$ and $(\Omega,\mathcal{F},\mathbb{P})$.
\end{proof}

\section{The dynamic programming equation}\label{Sec:dpe}
Now we are going to study the Hamilton Jacobi equation for problem (P) and design on this basis an optimal closed loop (feedback) controller $u$. To this end we associate to problem $(P)$ the optimal value function $V:[0,T]\times\R\times\R\to\R$
\begin{equation}
\label{eq:value}
V(t,x,y)=\inf_{u\in\U}\left\{\E\int_t^T X_1^2(s)f(X_1(s),u(s))\,ds+\E[ g(X_1(T))]\right\}
\end{equation}
subject to the control system
\begin{equation}
\label{sys:}
\begin{cases}
dX_1(s)=\mu X_1(s)\,ds+X_1\sqrt{u(s) X_2(s)}\,dW_1(s)\,,\quad s\in(t,T)\,,\\
dX_2(s)=k(\theta-X_2(s))\,ds+\sigma\sqrt{X_2(s)}\,dW_2(s)\,,\quad s\in(t,T)\,,\\
X_1(t)=x\,,\quad X_2(t)=y\,.
\end{cases}
\end{equation}
We shall assume $x>0$, $y\ge0$ and that conditions \eqref{eq:ic} holds. Then by Theorem~\ref{thm:existence} there is a martingale solution $X=(X_1,X_2)$ such that $X_1(s)\ge0$, $X_2(s)\ge 0$, $\forall s\in[0,T]$, $\mathbb{P}$-a.s.
We consider the Hamilton-Jacobi equation associated with problem (P), namely
\begin{equation}
\label{eq:phi}
\begin{split}
\varphi_t(t,x,y)+\mu x\varphi_x(t,x,y)+k(\theta-y)\varphi_y(t,x,y)+\frac{1}{2}\sigma^2y\varphi_{yy}(t,x,y)+\\+x^2G(x,y,\varphi_{xx}(t,x,y))=0\,,\quad t\in[0,T], \quad x,y\in\R\,,\\
\varphi(T,x,y)=g(x)\,,\quad \forall x\in\R\,,
\end{split}
\end{equation}
where $G\colon\R\times\R\times\R\to\R$ is given by
\begin{equation}
\label{def:G}
G(x,y,z)=\min_u \left\{\frac{1}{2}uyz+f(x,u) ; u\in[a,b]\right\}\,,\quad\forall x,y,z\in\R\,.
\end{equation}
It is well known (see e.g. \cite{Fleming,Oksendal}) that if $\varphi$ is a smooth (of class $C^{1,2}$ for instance) solution to \eqref{eq:phi} then the feedback controller
\begin{equation}
\label{def:opt}
u(t)=\phi(t,X_1(t),X_2(t))\,,\quad t\in[0,T]\,,
\end{equation}
where
\begin{equation}
\label{def:phi}
\phi(t,x,y)=\arg\min_u \left\{\frac{1}{2}uy\varphi_{xx}(t,x,y)+f(x,u); u\in[a,b]\right\}\,,
\end{equation}
is optimal in problem (P).

It should be mentioned that in general an equation of the form \eqref{eq:phi} does not have a classical solution. The best one can expect is a viscosity solution (see e.g. \cite{Soner}) which under our conditions is not unique and also it is not smooth enough to provide a feedback controller $\phi$ of the form \eqref{def:phi}. However, as shown below we can reduce \eqref{eq:phi} by a simple argument to a nonlinear parabolic equation for which one can prove the existence and uniqueness of a strong solution. Indeed, by setting
\begin{equation}
\label{def:p}
p(t,x,y)=\varphi_x(t,x,y)\,,\quad \forall t\in[0,T],\,x,y\in\R
\end{equation}
and differentiating in $x$, we transform \eqref{eq:phi} in the second order nonlinear parabolic equation
\begin{equation}
\label{eq:p}
\begin{cases}
p_t(t,x,y)+\mu (xp(t,x,y))_x+k(\theta-y)p_y(t,x,y)+\frac{1}{2}\sigma^2yp_{yy}(t,x,y)+\\\hspace{2.5cm}+(x^2G(x,y,p_{x}(t,x,y)))_x=0\,,\quad\forall t\in[0,T], \quad x,y\in\R\,,\\
p(T,x,y)=g_x(x)\,,\quad \forall x\in\R\,,\quad y\in\R\,,
\end{cases}
\end{equation}
with natural boundary conditions at $x=\pm\infty$, $y=\pm\infty$. In terms of $p$ the feedback controller $\phi$ in \eqref{def:opt} is expressed as
\begin{equation}
\phi(t,x,y)=\arg\min_{u} \left\{ \frac{1}{2}u y p_x(t,x,y)+f(x,u); u\in[a,b]\right\}\,.
\end{equation}

Taking into account that by \eqref{eq:pos} the state $X_2$ is in the half plane $y\ge0$, we see that the flow $t\mapsto(X_1(t),X_2(t))$ leaves invariant the domain $Q=\set{(x,y)\in\R^2; 0\le y<\infty}$ and so equation \eqref{eq:p} can be treated on this domain. For simplicity we shall restrict ourselves to the domain 
\[
Q=\set{x\in\R, \rho<y<M}=\R\times(\rho,M)
\]
where $M$ is sufficient large, but finite. In other words, we shall consider equation \eqref{eq:p} on domain $(0,T)\times Q$ with boundary value condition on $\partial Q$
\begin{equation}
\label{def:bc}
p(t,x,\rho)=0\,;\quad p(t,x,M)=0\,,\quad\forall x\in\R\,,\quad t\in[0,T]\,.
\end{equation}
By estimate \eqref{eq:boundX2} we see that for $M$ large enough
the above stochastic flow remains with a high probability in $Q$ and so we may treat problem \eqref{eq:p} in such a bounded domain.
We set $H=L^2(Q)$ with the standard norm $\norm{\cdot}_H$ and define the space 
\begin{equation}
\label{def:V}
V=\set{z\in H\cap H^1_{\text{loc}}(Q)\,;\,xz_x,z_y\in L^2(Q)\,;\,z(x,\rho)=z(x,M)=0\,,\quad\forall x\in\R}
\end{equation}
where $z_x$, $z_y$ are taken in sense of distributions on $Q$. The space $V$ is a Hilbert space with the norm
\[
||{z}||_V=\left(\int_Q (z^2+x^2z_x^2+z_y^2)\,dxdy\right)^\frac{1}{2}\,,\quad\forall z\in V\,.
\]
We have $V\subset H$ algebraically and topologically and denote by $V^\star$ the dual space of $V$ having $H$ as pivot space. Denote by $\prescript{}{V^\star}{(\cdot,\cdot)}_V$ the duality $(V,V^\star)$ and by $\norm{\cdot}_{V^\star}$ the dual norm of $V^\star$. On $H\times H$, $\prescript{}{V^\star}{(\cdot,\cdot)}_V$ is just the scalar product of $H$.
\begin{definition}\label{def:weaksol} The function $p:[0,T]\times Q\to\R$ is called weak solution to problem \eqref{eq:p}-\eqref{def:bc} if the following conditions hold
\begin{align}
\label{eq:pL2}
&p\in C([0,T];H)\cap L^2([0,T];V)\,,\quad\frac{dp}{d t}\in L^2([0,T];V^\star)\,,\\
&
\begin{aligned} 
&\frac{d}{dt}\int_Q p(t,x,y)\psi(x,y)\,dx\,dy+\int_Q(\mu (xp(t,x,y))_x+k(\theta-y)p_y(t,x,y))\psi(x,y)\,dx\,dy\\&\hspace{2cm}-\frac{\sigma^2}{2}\int_Q p_{y}(t,x,y)(y\psi(x,y))_y\,dx\,dy-\\&
\hspace{2cm}-\int_Q x^2G(x,y,p_{x}(t,x,y))\psi_x(x,y)\,dx\,dy=0\,\quad \forall \psi\in V\,, {\text{ a.e. }}t\in[0,T]\,,
\end{aligned}\\
&p(T,x,y)=g_x(x)\,,\quad\forall(x,y)\in  Q\,. \label{eq:pf}
\end{align}
\end{definition}
Taking into account \eqref{def:p} and Definition~\ref{def:weaksol} we say that $\varphi$ is a weak solution to \eqref{eq:phi} if
\begin{align}
&\begin{aligned}\label{eq:phiL2}
&\varphi\in L^2([0,T];L^2_\text{loc}(\R\times\R)),\, \varphi_x\in C([0,T];H)\cap L^2([0,T];V)\,,\\
&\frac{d\varphi_x}{d t}\in L^2([0,T];V^\star)\,,
\end{aligned}\\
&
\begin{aligned} \label{eq:phiint}
&\frac{d}{dt}\int_Q \varphi_x(t,x,y)\psi(x,y)\,dx\,dy+\int_Q(\mu (x\varphi_x(t,x,y))_x+k(\theta-y)\varphi_{xy}(t,x,y))\psi(x,y)\,dx\,dy\\&\hspace{2cm}-\frac{\sigma^2}{2}\int_Q \varphi_{xy}(t,x,y)(y\psi(x,y))_y\,dx\,dy-\\&
\hspace{2cm}-\int_Q x^2G(x,y,\varphi_{xx}(t,x,y))\psi_x(x,y)\,dx\,dy=0\,\quad \forall \psi\in V\,, \quad{\text{a.e. }}t\in[0,T]\,,
\end{aligned}\\
&
\begin{aligned}\label{eq:phif}
&\varphi(T,x,y)=g(x)\,,\quad\forall(x,y)\in  Q\,, \\
&\varphi_x(t,x,\rho)=\varphi_x(t,x,M)=0\,\quad \forall x\in\R, \,t\in[0,T]\,.
\end{aligned}
\end{align}
Clearly, if $p$ is a weak solution to \eqref{eq:p}-\eqref{def:bc} then by \eqref{def:p} the function
\[
\varphi(t,x,y)=\int_\infty^x p(t,\xi,y)\, d\xi\,,\quad(t,x,y)\in[0,T]\times\mathbb{R}\times(0,M)
\]
is a weak solution to \eqref{eq:phi} and conversely if $\varphi$ is a weak solution to \eqref{eq:phi} then $p$ given by \eqref{def:p} is a weak solution to \eqref{eq:p}. It should be said that $\varphi$ is unique until an additive function $\tilde\varphi=\tilde\varphi(t,y)$.

We have
\begin{theorem}
\label{thm:uws}
Let $g_x\in L^2(\R)$. Then there is a unique weak solution $p$ to problem \eqref{eq:p}-\eqref{def:bc}.
\end{theorem}
\begin{proof}
We can rewrite problem \eqref{eq:p}-\eqref{def:bc} as the backward infinite dimensional Cauchy problem
\begin{equation}
\label{eq:A}
\begin{aligned}
&\frac{d}{d t} p(t)-Ap(t)=0\,,\quad \text{a.e. } t\in(0,T)\,,\\
&p(T)=g_x\,,
\end{aligned}
\end{equation}
where $A:V\to V^\star$ is the nonlinear operator defined by
\begin{equation}
\label{def:A}
\begin{split}
(Az,\psi)=-\int_Q(\mu (xz)_x+k(\theta-y)z_y)\psi(x,y)\,dx\,dy+\frac{\sigma^2}{2}\int_Q z_{y}(y\psi(x,y))_y\,dx\,dy+\\
\hspace{2cm}+\int_Q x^2G(x,y,z_x)\psi_x(x,y)\,dx\,dy\,,\quad t\in[0,T], \quad z,\psi\in V\,.\\
\end{split}
\end{equation}

In order to apply the standard existence theory for the Cauchy problem \eqref{eq:A} (see e.g. \cite[p. 177]{Barbu}) we need to check the following properties for $A(t)$.
\begin{enumerate}[(I)]
\item \label{j} There is $\alpha_1\ge0$ such that
\[
\prescript{}{V^\star}{(Az-A\bar z,z-\bar z)}_V\ge-\alpha_1\norm{z-\bar z}^2_{H}\,,\quad\forall z,\bar z\in V\,.
\]
\item \label{jj} There is $\alpha_2>0$ such that
\[
\norm{Az}_{V^\star}\le\alpha_2\norm{z}_V\,,\quad\forall z\in V, t\in(0,T)\,.
\]
\item \label{jjj} There are $\alpha_3>0,\,\alpha_4\ge0$ such that
\[
\prescript{}{V^\star}{(A(t)z,z)}_V\ge\alpha_3\norm{z}^2_V-\alpha_4\norm{z}^2_H,\quad\forall z\in V, t\in(0,T)\,.
\]
\end{enumerate}
To check (I) we note that by \eqref{def:G} we have
\begin{equation}
\label{eq:Gfstar}
\begin{aligned}
G(x,y,z)&=-\sup_u\left\{-\frac{1}{2}uyz-f(x,u)\,; u\in[a,b]\right\}\\
&=-\tilde f^\star\left(x,-\frac{1}{2}yz\right)\quad\forall (x,y)\in Q, z\in\R\,,
\end{aligned}
\end{equation}
where $\tilde f^\star(x,v)$ is the convex conjugate of function $v\xrightarrow[]{\tilde f} f(x,v)+I_{[a,b]}(v)$ that is $\tilde f^\star (x,q)=\sup_v\set{qv-\tilde f(x,v); \quad v\in[a,b]}$ and $I_{[a,b]}(v)=0$ if $v\in[a,b]$, $I_{[a,b]}=+\infty$ otherwise. 
Given a convex, lower semicontinuous function $h\colon\R\to]-\infty,+\infty]$ denote by $\partial h(v)=h_v(v)$ the subdifferential of $h$ at $v$, that is
\[
\partial h(v)=\set{\eta\in\R : \eta(v-\bar v)\ge h(v)-h(\bar v)\,\quad\forall\bar v\in\R}\,,
\]
and recall that if $h^\star$ is the conjugate of $h$ then $\partial h^\star (q)=(\partial h)^{-1}(q)$, $\forall q\in\R$ and so $h^\star_q=\partial_qh^\star(\alpha q)=\alpha\partial h^\star(\alpha q)$, $\forall \alpha\in\R$, $q\in\R$.
This yields
\begin{equation}
\label{eq:Gz}
G_z(x,y,z)=\frac{1}{2}y\tilde f^\star_v\left(x,-\frac{1}{2}yz\right),\quad\forall (x,y)\in Q, z\in\R\,,
\end{equation}
where $G_z$ is the subdifferential of function $z\mapsto G(x,y,z)$.
On the other hand, we have
\begin{equation}
\label{eq:fsv1}
\tilde f^\star_v(x,v)=\left(f_u(x,\cdot)+N_{[a,b]}\right)^{-1}(v),\quad\forall v\in\R
\end{equation}
where $N_{[a,b]}(v)\subset2^\R$ is the normal cone to $[a,b]$ in $v$, that is
\begin{equation}
\label{def:N}
N_{[a,b]}(v)=\begin{cases}
\R^-\quad&\text{if $v=a\,,$}\\
0\quad&\text{if $a<v<b\,,$}\\
\R^+\quad&\text{if $v=b\,.$}
\end{cases}
\end{equation}
This yields
\begin{equation}
\label{eq:fsv2}
\tilde f^\star_v(x,v)\in[a,b]\,,\quad\forall x,v\in\R\,,
\end{equation}
and since $y\in[0,M]$, $a>0$, by \eqref{eq:Gz} we see that $G_z\ge 0$ 
and so the function $z\to G(x,y,z)$ is monotonically nondecreasing. By \eqref{def:A} it follows via integrations by part that (I) holds with some suitable $\alpha_1$.

To prove (II) we note that
\[
\int_Q x^2G(x,y,z_x)\psi_x(x,y)\,dx\,dy\le\norm{\psi}_V\left(\int_Q|xG(x,y,x)|^2\,dx\,dy\right)^\frac{1}{2}\,,
\]
while by \eqref{eq:Gz} and \eqref{eq:fsv2} we have that for all $(x,y)\in Q$, $v\in\R$,
\[
|G(x,y,v)|=|vG_v(x,y,\xi_v)|\le\frac{1}{2}Mb|v|\,,
\]
because by hypothesis (i) and \eqref{eq:Gfstar} it follows that $G(x,y,0)=0$.
Hence $\forall \psi,z\in V$
\[
\left\lvert\int_Q  x^2G(x,y,z_x)\psi_x(x,y)\,dx\,dy\right\lvert\le\frac{1}{2}Mb\int_Q  x^2\abs{z_x}\abs{\psi_x}\,dx\,dy\le C\norm{\psi}_V\norm{z}_V\,.
\]
Similarly, we have
\[
\left\lvert\int_Qz_y(y\psi)_y\,dx\,dy\right\lvert\le C\norm{\psi}_V\norm{z}_V\quad\forall \psi,z\in V
\]
and by \eqref{def:A} it follows (II).

Finally, taking into account that by \eqref{eq:Gz}-\eqref{eq:fsv1}, $G_z(x,y,\xi_z)\ge \frac{1}{2}ay>\frac{1}{2}a\rho$, we have
\begin{equation}
\label{eq:xzx}
\int_Q  x^2G(x,y,z_x)z_x\,dx\,dy=\int_Q  x^2G_z(x,y,\xi_z)z^2_x\,dx\,dy\ge \frac{1}{2}a\rho\int_Q  x^2z^2_x\,dx\,dy\,.
\end{equation}
We note also the inequalities
\begin{multline*}
\int_Qz_y(yz)_y\,dx\,dy = \int_Q\left( y |z_y|^2+\frac{1}{2}(z^2)_y\right)\,dx\,dy\ge\\
\ge\int_Q\left( y |z_y|^2\right)\,dx\,dy\ge\rho\int_Q\left(  |z_y|^2\right)\,dx\,dy
\end{multline*}
and
\begin{align*}
-\int_Q \left((xz)_x+k(\theta-y)z_y\right)z\,dx\,dy&=-\int_Q \left(\frac{1}{2}x(z^2)_x+z^2+\frac{k}{2}(\theta-y)(z^2)_y\right)\,dx\,dy\\
&\ge C\int_Q\left(z^2+y^2\right)\,dx\,dy\,.
\end{align*}
Here we have integrated by parts.
Together with \eqref{eq:xzx} the latter implies (III) as claimed.

Then we infer that the Cauchy problem \eqref{eq:A} has a unique solution $p$ which satisfies \eqref{eq:pL2} and this completes the proof of the theorem.
\end{proof}

An alternative approach to treat equation \eqref{eq:A} is the so called \emph{semigroup approach} we briefly present below.
 
We set $q(t)=p(T-t)$ and rewrite \eqref{eq:A} as the forward Cauchy problem
\begin{align}
\label{eq:Cpp}
&\frac{dq}{dt}+Aq=0\,,\quad t\in(0,T)\,,\\
&q(0)=g_x\,.
\end{align}
Define the operator $A_H:\mathcal{D}(A_H)\subset H \to H$ as
\[
A_Hq=Aq\,,\quad\forall q\in\mathcal{D}(A_H)=\set{q\in V : Aq\in H}\,.
\]
By (I)-(III) it follows  via Minty-Browder theory (see e.g. \cite{Barbu}) that the operator $A_H$ is quasi-m-accretive in $H\times H$, that is, there is $\eta_0\in\R$ such that
\[
(A_Hq-A_H\bar q,q-\bar q)\ge-\eta_0\norm{q-\bar q}_H^2\,,\quad\forall q,\bar q\in\mathcal{D}(A_H)\,,
\]
and
\[
R(\eta I+A_H)=H\,,\forall \eta>\eta_0\,,
\]
where $I$ is the identity operator and  $R$ indicates the range.
As a matter of fact by (I)-(III) it follows that $R(\eta I+A_H)\supset V^\star$ which clearly implies the latter.
Then, by the standard existence theory for the Cauchy problem associated with non-linear quasi-m-accretive  operators in Hilbert spaces, if $g_x\in\mathcal{D}(A_H)$ then there is a unique strong solution $q\in W^{1,\infty}([0,T];H)=\set{q\in L^\infty([0,T],H); \frac{dq}{dt}\in L^\infty([0,T],H)}$ to equation \eqref{eq:Cpp}, that is
\begin{equation}
\label{eq:Aq}
\begin{aligned}
&q\in L^2([0,T],V)\,,\quad Aq(t)\in L^\infty([0,T];H)\,,\\
&\frac{d^+}{dt} q(t)+A_H(q(t))=0\,,\quad t\in[0,T[\,,\\
&q(0)=g_x=q_0\,.
\end{aligned}
\end{equation}
Moreover, $q$ is expressed by the exponential Crandall\&Liggett formula
\begin{equation}
q(t)=\lim_{n\to\infty}\left(1+\frac{t}{n}A\right)^{-n}q_0
\end{equation}
 uniformly on $[0,T]$ (see e.g. \cite{Barbu}, p. 139).
 This means that the solution $q$ is the limit of the finite difference scheme $$q(t)=\lim_{h\to0}q_h(t)\quad\forall t\in[0,T]$$
where
\begin{equation}
\label{eq:qh}
\begin{aligned}
& q_h(t)=q^i_h\,\quad t\in[ih, (i+1)h)\,,\\
& q^{i+1}_h+hAq^{i+1}_h=q^i_h\,,\quad i=0,1,\dots,N=\left[\frac{T}{h}\right]-1\,,\\
& q^0_h=q_0\,.
\end{aligned}
\end{equation}
For $q_0\in H$ the function $q$ given by \eqref{eq:qh} is a generalized mild solution to \eqref{eq:Cpp}.
This scheme may be used to numerical approximation of equation \eqref{eq:p}. Moreover, this reveals that for $g$ regular the solution $p$ of equation \eqref{eq:p} is locally in $H^2(Q)$.

We note also that $q_0\to q(t)$ (for instance in $H^2(Q)$) is a semigroup of quasi contraction on $H$, that is
\[
\norm{q(t,q_0)-q(t,\bar q_0)}_H\le\exp(\eta_0 t)\norm{q_0-\bar q_0}_H\quad\forall t\ge0\,,\quad q_0,\bar q_0\in H\,.
\]

Now coming back to equation~\eqref{eq:phi}, we get by Theorem~\ref{thm:uws} 
\begin{corollary}\label{cor:us}
Let $g_x\in L^2(\R)$. Then equation \eqref{eq:phi},\eqref{eq:phif} has a weak solution in the sense of \eqref{eq:phiL2}-\eqref{eq:phiint}. This solution is unique up to an additive function $\tilde\varphi\equiv \tilde\varphi(t,y)$.
\end{corollary}

\begin{remark}
In literature, dynamic programming equations of the form~\eqref{eq:phi} have been treated so far in the framework of viscosity solutions (see, e.g., \cite{Fleming}, \cite{Soner}, \cite{Heston}). 
The main advantage of Theorem~\ref{thm:uws} and respectively of Corollary~\ref{cor:us} is that are obtained here a weak solution $\varphi$ which is however sufficiently regular to represent the feedback controller $u$ of problem (P) in explicit terms \eqref{def:phi}.
\end{remark}
\section{The optimal feedback controller}\label{sec:optimalfeedback}
Now coming back to \eqref{def:opt}-\eqref{def:phi} one might suspect that the feedback controller $u^\star$ defined by
\begin{equation}
\label{eq:us}
u^\star(t)=\phi(t,X_1(t),X_2(t))\quad t\in(0,T)\,,
\end{equation}
where
\begin{equation}
\phi(t,x,y)=\arg\min_{u\in[a,b]}\left\{\frac{1}{2}uy p_x(t,x,y)+f(x,u)\right\}\,,\quad\forall(x,y)\in Q\,,
\end{equation}
is optimal in problem (P) for $(X_1,X_2)\in(0,T)\times Q$, where $p$ is the weak solution to equation \eqref{eq:p}-\eqref{def:bc}.
Since $p_x\in L^2(Q)$, the mapping $\phi$ is well defined and (see \eqref{eq:fsv1}) we have
\begin{equation}
\label{eq:phi2}
\phi(t,x,y)=\left(f_u(x,\cdot)+N_{[a,b]}\right)^{-1}\left(-\frac{1}{2}yp_x(t,x,y)\right),\quad\forall t\in[0,T]\,,(x,y)\in Q\,,
\end{equation}
where $f_u(x,\cdot)$ is the subdifferential of function $u\mapsto f(x,u)$.
The corresponding closed loop system \eqref{sys:model} is
\begin{equation}
\label{sys:cl}
\begin{cases}
dX_1=\mu X_1\,dt+X_1\sqrt{\phi(t,X_1,X_2) X_2}\,dW_1\,,\quad t\in(0,T)\,\\
dX_2=k(\theta-X_2)\,dt+\sigma\sqrt{X_2}\,dW_2\,,\quad t\in(0,T)\,,\\
X_1(0)=X_1^0\,,\quad X_2(0)=X_2^0\,.
\end{cases}
\end{equation}
The existence of a strong solution $(X_1,X_2)$ to \eqref{sys:cl} would imply by a standard computation that the map $u^\star=\phi(t,X_1,X_2)$ is indeed an optimal feedback controller for problem (P). However the existence of a strong solution for \eqref{sys:cl} is a delicate problem and the best one can expect in this case is a martingale solution. To this end we assume in addition  beside above hypotheses that
\begin{enumerate}[(ii)]
\item $u\mapsto f_u(x,u)$ is strictly monotone.
\end{enumerate}
\begin{theorem}
\label{thm:ocsm}
Assume that \textnormal{(i),(ii)} and \eqref{eq:ic} hold. Then there is a martingale solution $(X_1,X_2)$ to equation \eqref{sys:cl}.
\end{theorem}
\begin{proof}
The proof is similar to that of Theorem~\ref{thm:existence} and so it will be sketched only.
Consider the map $\psi\colon[0,T]\times\R\times\Omega \to \R$
\[
\psi(t,x)=x\sqrt{\phi(t,x,X_2(t))X_2(t)}\,,\quad t\in[0,T],\,x\in\R\,.
\]
By \eqref{eq:phi2} it follows that for all $R>0$, $\abs{x}+\abs{\bar x}\le R$, $\mathbb{P}$-a.s.
{\small
\begin{equation}
\label{eq:bounds}
\begin{aligned}
\abs{\psi(t,x)-\psi(t, \bar x)}&=\abs{x\sqrt{\phi(t,x,X_2(t))X_2(t)}-\bar x\sqrt{\phi(t,\bar x,X_2(t))X_2(t)}}\\
&\le\sqrt{X_2(t)}\left( \abs{x-\bar x}\sqrt{\phi(t,x,X_2(t))} +\abs{\bar x}\sqrt{\abs{\phi(t, x,X_2(t))-\phi(t,\bar x,X_2(t))}}\right)\\
&\le C\sqrt{X_2(t)}\abs{x-\bar x}+R \sqrt{X_2(t)}\sqrt{L_\phi\frac{1}{2}X_2(t)\abs{p_x(t,x,X_2(t)) +p_x(t, x,X_2(t))}}\\
&\le C\sqrt{X_2(t)}\abs{x-\bar x}+R X_2(t)\sqrt{L_\phi}\sqrt{L^1_R\abs{x-\bar x}}\\
&\le C\sqrt{X_2(t)}\abs{x-\bar x}+K X_2(t) \sqrt{\abs{x-\bar x}}
\end{aligned}
\end{equation}
}
and 
\begin{equation}
\label{eq:bounds2}
\abs{\psi(t,x)}\le C\abs{x}\sqrt{X_2(t)}\,.
\end{equation}
Indeed, by (ii) the map $\left(f_u(x,\cdot)+N_{[a,b]}\right)^{-1}$ is Lipschitz on $\R$ while
\[
\abs{p_x(t,x)-p_x(t,\bar x)}\le L^1_R\abs{x-\bar x}\,,\quad\text{for }\abs{x}+\abs{\bar x}\le R\,.
\]
The latter is a consequence of high order regularity of solutions to quasilinear parabolic equation (see \cite{Ladyzenskaya}, Theorem 6.1, p 452) which implies that
\[
p\in H^{2+\alpha, 1+\frac{\alpha}{2}}((0,T)\times Q_R)\quad\forall R>0
\]
where $\alpha\in(0,1)$ and $Q_R=(0,R)\times(0,M)$.
On the other hand again by \eqref{eq:phi2} we see that $\phi(t,x,y)\in[a,b]$ $\forall x,y\in\R$, $t\in[0,T]$, and so \eqref{eq:bounds} follows.

We approximate \eqref{sys:cl} by (see \eqref{eq:eps})
\begin{equation}\label{sys:cl2}
\begin{aligned}
&dX^\e_1=\mu X^\e_1\,dt-X^\e_1\sqrt{\phi(t,X^\e_1,X^\e_2) X^\e_2}\,dW_1\,,\\
&dX^\e_2=k(\theta-X^\e_2)\,dt+\sigma\frac{X^\e_2}{\sqrt{\abs{X^\e_2}+\e}}\,dW_2\,,\\
&X^\e_1(0)=X_1^0\,,\quad X^\e_2(0)=X_2^0\,.
\end{aligned}
\end{equation}
Arguing as in the proof of Theorem~\ref{thm:existence} it follows by \eqref{eq:bounds}-\eqref{eq:bounds2} that \eqref{sys:cl2} has a unique solution $(X^\e_1,X^\e_2)$, $X^\e_1,X^\e_2\ge 0$, $\mathbb{P}$-a.s.
Moreover, one obtains also in this case estimates \eqref{supXe}-\eqref{stXe} and so by the Skorohod theorem it follows as above the existence of a martingale solution $(\tilde X_1,\tilde X_2)$ satisfying system~\eqref{sys:cl} in a probability space $(\tilde\Omega,\mathcal{\tilde F},\mathbb{\tilde P},\tilde W_1,\tilde W_2)$.
\end{proof}

\begin{remark}
Roughly speaking Theorem~\ref{thm:ocsm} amounts to saying that there is a probability space $(\tilde\Omega,\mathcal{\tilde F},\mathbb{\tilde P},\tilde W_1,\tilde W_2)$ where the closed loop system corresponding feedback controller $\tilde u^\star=\phi(t,\tilde X_1,\tilde X_2)$ has a solution $(\tilde X_1,\tilde X_2)$. This means that the feedback controller $u^*$ is admissible in problem (P) though it is not clear if it is optimal. However this is a suboptimal feedback controller. Indeed, if one constructs in a similar way the feedback controller $u_\e=\phi_\e(X^\e_1,X^\e_2)$ for problem (P) with state system~\eqref{sys:cl2} and $\phi_\e$ is the solution to the corresponding equation \eqref{eq:phiL2}-\eqref{eq:phif} then $u_\e$ is optimal for the approximating optimal control problem and it is convergent in law in the above sense to a controller $u^*$ as found above.
\end{remark}

\begin{example}
Let $f(x,u)\equiv0$. Then equation \eqref{eq:p} reduces to
\begin{equation}
\label{eq:pex2}
\begin{aligned}
&p_t+\mu(xp)_x+k(\theta-y)p_y+\frac{\sigma^2}{2}yp_{yy}+ y\left(x^2(aH(p_x)+bH(-p_x))\right)_x=0\,,\: x\in\R\,,\, y\in(0,M)\,,\\
&p(T,x,y)=g_x(x)\,,
\end{aligned}
\end{equation}
while by \eqref{eq:phi2} the feedback controller $u^\star$ (see \eqref{eq:us}) is given by
\begin{equation}
u^\star(t)=
\begin{cases}
a\quad&\text{if $p_x(t,X_1(t),X_2(t))>0$,}\\
b\quad&\text{if $p_x(t,X_1(t),X_2(t))<0$,}\\
\in(a,b)\quad&\text{if $p_x(t,X_1(t),X_2(t))=0$.}
\end{cases}
\end{equation}
In \eqref{eq:pex2} $H$ is the Heaviside function.

It should be said however that in this case Theorem~\ref{thm:ocsm} does not apply and so parabolic equation \eqref{eq:pex2} has only a weak solution $p$ in the sense of \eqref{eq:pL2}-\eqref{eq:pf}  and so $p$ is not sufficiently regular to assume existence of the closed loop system \eqref{eq:pex2}. 
However, if for almost all $\omega\in\Omega$ the set $\Sigma=\set{t : p_x(t, X^\star_1(t),X^\star_2(t))=0}$ is finite (as it is expected to be) then the controller $u^\star$ is a bang-bang controller with $\Sigma$ as set of switch points. This fact might be lead to a simplification of control problem (P) by replacing the set $\mathcal{U}_0$ of admissible controllers $u\colon[0,T]\to\R$ by $$\mathcal{\tilde U}_0=\set{u\colon[0,T]\to\R\,,\mathcal{F}_t-\text{adapted},\, u(t)=\sum_{i=0}^{N-1} v_i\chi_{[t_i,t_{i+1}]}(t)}\,. $$
Here $t_0=0<t_1<t_2<\dots<t_N=T$ is a given partition of interval $[0,T]$ while $v_i\colon\Omega\to\R$ are $\mathcal{F}_{t_i}$-measurable functions. More precisely we can consider the stochastic optimal control problem 
\begin{equation}
\text{Minimize } \E g(X(T))
\end{equation}
subject to \eqref{sys:model} and $u\in\mathcal{\tilde U}_0$ where $v_j\in[a,b]$, $\forall j=0,1,\dots,N-1$.
\end{example}

\section{Concluding comments}\label{Sec:conclusion}
In the present paper we have studied an optimization problem where the performance criterion is a function of a variation of the celebrated Heston model, where an additional control process has been taken into account, namely
\begin{gather*}
\min_{u\in\mathcal{U}}\E\int_0^T X_1^2(t)f(X_1(t),u(t))\,dt+\E g(X_1(T))\,,\\
s.t.
\begin{cases}
dX_1=\mu X_1\,dt+X_1\sqrt{u X_2}\,dW_1\,,\quad t\in(0,T)\,\\
dX_2=k(\theta-X_2)\,dt+\sigma\sqrt{X_2}\,dW_2\,,\quad t\in(0,T)\,\\
X_1(0)=X_1^0\,,\quad X_2(0)=X_2^0\,.
\end{cases}
\end{gather*}
Under suitable conditions, we have established the well posedness of the control problem via an approximation-technique. Moreover,  we study the associated HJB equation, also discussing the existence of an optimal feedback controller for this problem.


\begin{thebibliography}{9}
\bibitem{Barbu}Barbu, Viorel. \emph{Nonlinear Differential Equations of Monotone Types in Banach spaces.} Springer Science \& Business Media, 2010.
\bibitem{Barbuetal} Barbu, Viorel, Stefano Bonaccorsi, and Luciano Tubaro. ``Stochastic differential equations with variable structure driven by multiplicative Gaussian noise and sliding mode dynamic." \emph{Mathematics of control signals and system}, 26 (2016).


\bibitem{Barndorff} Barndorff-Nielsen, Ole E. ``Econometric analysis of realized volatility and its use in estimating stochastic volatility models." \emph{Journal of the Royal Statistical Society: Series B (Statistical Methodology)} 64.2 (2002): 253-280.

\bibitem{ChangRong} Chang, Hao, and Xi-min Rong. 
``An investment and consumption problem with CIR interest rate and stochastic volatility" \emph{Abstract and Applied Analysis}, (2013): art. no. 219397

\bibitem{CIR} Cox, John C., Jonathan E. Ingersoll Jr, and Stephen A. Ross. ``A theory of the term structure of interest rates." \emph{Econometrica: Journal of the Econometric Society} (1985): 385-407.

\bibitem{CordoniDipersioBackward}Cordoni, Francesco, and Luca Di Persio.  ``Backward stochastic differential equations approach to hedging, option pricing, and insurance problems"
\emph{International Journal of Stochastic Analysis}, (2014) : art. no. 152389

\bibitem{CordoniDipersioVasicek} Cordoni, Francesco, and Luca Di Persio.  ``Invariant measure for the Vasicek interest rate model in the Heath-Jarrow-Morton-Musiela framework"
\emph{Infinite Dimensional Analysis, Quantum Probability and Related Topics}, 18 (3), (2015) : art. no. 1550022


\bibitem{DaPrato} Da Prato, Giuseppe, and Jerzy Zabczyk. \emph{Stochastic equations in infinite dimensions.} Cambridge university press, 2014.
\bibitem{Grzelak} Grzelak, Lech A., and Cornelis W. Oosterlee.``On the Heston model with stochastic interest rates."
\emph{ SIAM Journal on Financial Mathematics}, 2 (1), (2011): 255-286.

\bibitem{Fleming} Fleming, Wendell H., and Raymond W. Rishel. \emph{Deterministic and stochastic optimal control.} Vol. 1. Springer Science \& Business Media, 2012.


\bibitem{Soner} Fleming, Wendell H., and Halil Mete Soner. \emph{Controlled Markov processes and viscosity solutions}. Vol. 25. Springer Science \& Business Media, 2006.
\bibitem{HaastrechtPelsser} van Haastrecht, Alexander, and Antoon Pelsser. 
``Generic pricing of FX, inflation and stock options under stochastic interest rates and stochastic volatility"
 \emph{Quantitative Finance}, 11 (5), (2011): 665-691.

\bibitem{Heston} Heston, Steven L. ``A closed-form solution for options with stochastic volatility with applications to bond and currency options." \emph{Review of financial studies} 6.2 (1993): 327-343.
\bibitem{Hull}Hull, John, and Alan White. ``The pricing of options on assets with stochastic volatilities." \emph{The journal of finance} 42.2 (1987): 281-300.
\bibitem{Ladyzenskaya} Lady\v{z}enskaja, Olga Aleksandrovna, Vsevolod Alekseevich Solonnikov, and Nina N. Ural'ceva. Linear and quasi-linear equations of parabolic type. Vol. 23. American Mathematical Soc., 1988.

\bibitem{Jones} Jones, Christopher S.``The dynamics of stochastic volatility: Evidence from underlying and options markets"
\emph{ Journal of Econometrics}, 116 (1-2), (2003): 181-224. 

\bibitem{Kraft} Kraft, Holger.``Optimal portfolios and Heston's stochastic volatility model: an explicit solution for power utility" \emph{Quantitative Finance}, vol. 5, no. 3 (2005): 303–313. 

\bibitem{Mandelbrot} Mandelbrot, Benoit. ``New methods in statistical economics." \emph{Journal of political economy} 71.5 (1963): 421-440.
\bibitem{Oksendal}Oksendal, Bernt. \emph{Stochastic differential equations: an introduction with applications.} Springer Science \& Business Media, 2013.
\bibitem{Scott} Scott, Louis O. "Option pricing when the variance changes randomly: Theory, estimation, and an application." \emph{Journal of Financial and Quantitative analysis} 22.04 (1987): 419-438.

\bibitem{Wiggins} Wiggins, James B. ``Option values under stochastic volatility: Theory and empirical estimates." \emph{Journal of financial economics} 19.2 (1987): 351-372.
\bibitem{YiLiViensZeng} Yi, Bo, et al.
"Robust optimal control for an insurer with reinsurance and investment under Heston's stochastic volatility model"
 Insurance: Mathematics and Economics, 53 (3): (2013) pp. 601-614.

\end{thebibliography}
\end{document}